\documentclass[a4paper, 12pt]{amsart}

\usepackage{amsmath}
\usepackage{amssymb}
\usepackage{ascmac}
\usepackage{amsthm}

\makeatletter

\@addtoreset{equation}{section}
\makeatother 

\theoremstyle{plain}
\newtheorem{thm}{Theorem}[section]
\newtheorem*{thm*}{Theorem}
\newtheorem{prop}[thm]{Proposition}
\newtheorem{lem}[thm]{Lemma}
\newtheorem{cor}[thm]{Corollary}

\theoremstyle{definition}

\theoremstyle{remark}
\newtheorem{rem}[thm]{Remark}

\renewcommand{\epsilon}{\varepsilon}

\newcommand{\ric}{\operatorname{Ric}}

\newcommand{\cut}{\mathrm{Cut}\,}

\newcommand{\bm}{\partial M}

\newcommand{\second}{{\rm I\hspace{-.01em}I}}

\usepackage{latexsym}

\title[Yau and Souplet-Zhang type gradient estimates]{Yau and Souplet-Zhang type gradient estimates\\ on Riemannian manifolds with boundary\\ under Dirichlet boundary condition}

\author{Keita Kunikawa}
\address{Cooperative Faculty of Education, Utsunomiya University, 350 Mine-Machi, Utsunomiya, 321-8505, Japan}
\email{kunikawa@cc.utsunomiya-u.ac.jp}

\author{Yohei Sakurai}
\address{Department of Mathematics, Saitama University, 255 Shimo-Okubo, Sakura-ku, Saitama-City, Saitama, 338-8570, Japan}
\email{ysakurai@rimath.saitama-u.ac.jp}

\subjclass[2010]{Primary 53C20; Secondly 31C05, 35K05, 35B40, 58J35}
\keywords{Harmonic function; Heat equation; Gradient estimate; Liouville theorem}
\date{July 28, 2021}

\begin{document}
\maketitle

\begin{abstract}
In this paper,
on Riemannian manifolds with boundary,
we establish a Yau type gradient estimate and Liouville theorem for harmonic functions under Dirichlet boundary condition.
Under a similar setting,
we also formulate a Souplet-Zhang type gradient estimate and Liouville theorem for ancient solutions to the heat equation.
\end{abstract}

\section{Introduction}

\subsection{Elliptic and parabolic gradient estimates}

In geometric analysis,
one of the main problems is to estimate the gradient of solutions to elliptic or parabolic equations on Riemannian manifolds under curvature bounds.
Yau \cite{Y} has established the following gradient estimate for harmonic functions (see \cite{Y}, and also \cite{CY}):
\begin{thm}[\cite{Y}, \cite{CY}]\label{thm:Yau}
Let $(M,g)$ be an $n$-dimensional, complete Riemannian manifold (without boundary).
For $K \geq 0$,
we assume $\ric_{M} \geq -(n-1)K$.
For $x_0\in M$,
let $u:B_{R}(x_0)\to (0,\infty)$ be a positive harmonic function.
Then we have
\begin{equation*}
\frac{\Vert \nabla u \Vert}{u}\leq C_n\left(\frac{1}{R}+\sqrt{K}  \right)
\end{equation*}
on $B_{R/2}(x_0)$,
where $C_n$ is a positive constant depending only on $n$.
\end{thm}

Inspired by Theorem \ref{thm:Yau} and the work of Hamilton \cite{H},
Souplet-Zhang \cite{SZ} have obtained the following parabolic version (see \cite[Theorem 1.1]{SZ}):
\begin{thm}[\cite{SZ}]\label{thm:SZ}
Let $(M,g)$ be an $n$-dimensional, complete Riemannian manifold (without boundary).
For $K \geq 0$,
we assume $\ric_{M} \geq -(n-1)K$.
For $x_0\in M$,
let $u$ be a positive solution to the heat equation on
\begin{equation*}
Q_{R,T}(x_0):=B_{R}(x_0)\times [-T,0].
\end{equation*}
For $A>0$, we assume $u< A$.
Then we have
\begin{equation*}
\frac{\Vert \nabla u \Vert}{u}\leq C_n \left( \frac{1}{R}+\frac{1}{\sqrt{T}}+\sqrt{K}  \right) \left( 1+\log \frac{A}{u}  \right)
\end{equation*}
on $Q_{R/2,T/4}(x_0)$.
\end{thm}

\subsection{Gradient estimates on manifolds with boundary}
The aim of this article is to formulate analogues of Theorems \ref{thm:Yau} and \ref{thm:SZ} on Riemannian manifolds with (smooth) boundary under Dirichlet boundary condition.
Let $(M,g)$ be a complete Riemannian manifold with boundary $\bm$.
For $R>0$,
we denote by $B_{R}(\bm)$ the $R$-neighborhood of $\bm$.
We also denote by $H_{\bm}$ the infimum of the mean curvature over $\bm$,
where the mean curvature is defined as the trace of the shape operator for the unit inner normal vector on $\bm$. 
For harmonic functions with Dirichlet boundary condition (i.e., it is constant on the boundary),
we prove the following gradient estimate on a neighborhood of the boundary under a lower Ricci curvature bound, and a lower mean curvature bound for the boundary:
\begin{thm}\label{thm:main1}
Let $(M,g)$ be an $n$-dimensional, complete Riemannian manifold with compact boundary.
For $K \geq 0$,
we assume $\ric_{M} \geq -(n-1)K$ and $H_{\partial M}\geq -(n-1)\sqrt{K}$.
Let $u:B_{R}(\partial M)\to (0,\infty)$ be a positive harmonic function with Dirichlet boundary condition.
We assume that its derivative $u_{\nu}$ in the direction of the outward unit normal vector $\nu$ is non-negative over $\bm$.
Then we have
\begin{equation*}
\frac{\Vert \nabla u \Vert}{u}\leq C_n\left(\frac{1}{R}+\sqrt{K}  \right)
\end{equation*}
on $B_{R/2}(\partial M)$.
\end{thm}

Further,
for solutions to the heat equation with Dirichlet boundary condition,
we obtain the following gradient estimate under the non-negativity of mean curvature on the boundary:
\begin{thm}\label{thm:main2}
Let $(M,g)$ be an $n$-dimensional, complete Riemannian manifold with compact boundary.
For $K \geq 0$,
we assume $\ric_{M} \geq -(n-1)K$ and $H_{\partial M}\geq 0$.
Let $u$ be a positive solution to the heat equation on
\begin{equation*}
Q_{R,T}(\bm):=B_{R}(\bm)\times [-T,0].
\end{equation*}
For $A>0$,
let us assume $u<A$.
We further assume that $u$ satisfies the Dirichlet boundary condition $($i.e., $u(\cdot,t)|_{\bm}$ is constant for each fixed $t\in [-T,0]$$)$, and $u_{\nu}\geq 0$ and $\partial_t u\leq 0$ over $\bm \times [-T,0]$.
Then we have
\begin{equation*}
\frac{\Vert \nabla u \Vert}{u}\leq C_n \left( \frac{1}{R}+\frac{1}{\sqrt{T}}+\sqrt{K}  \right) \left( 1+\log \frac{A}{u}  \right)
\end{equation*}
on $Q_{R/2,T/4}(\bm)$.
\end{thm}

\subsection{Liouville theorems}

Theorem \ref{thm:main1} together with the standard argument leads to the following Liouville theorem:
\begin{cor}\label{cor:Liouville1}
Let $(M,g)$ be a complete Riemannian manifold with compact boundary.
We assume $\ric_{M} \geq 0$ and $H_{\partial M}\geq 0$.
Let $u:M\to (0,\infty)$ be a positive harmonic function with Dirichlet boundary condition.
We assume $u_{\nu}\geq 0$ over $\bm$.
Then $u$ is constant.
\end{cor}

\begin{rem}\label{rem:assumption}
Without the assumption $u_{\nu}\geq 0$,
this type of Liouville theorem does not hold.
Actually,
if $M=[0,\infty)$ and $u(x)=1+x$,
then it is a non-trivial positive harmonic function.
\end{rem}

Theorem \ref{thm:main2} together with a similar argument of the proof of \cite[Theorem 1.2]{SZ} tells us the following:
\begin{cor}\label{cor:Liouville2}
Let $(M,g)$ be a complete Riemannian manifold with compact boundary.
We assume $\ric_{M} \geq 0$ and $H_{\partial M}\geq 0$.
Then we have the following:
\begin{enumerate}\setlength{\itemsep}{+0.7mm}
\item Let $u:M\times (-\infty,0]\to (0,\infty)$ be a positive ancient solution to the heat equation with Dirichlet boundary condition,
and $u_{\nu}\geq 0$ and $\partial_t u\leq 0$ over $\bm \times (-\infty,0]$.
If
\begin{equation*}
u(x,t)=\exp\left[o\left(\rho_{\bm}(x)+\sqrt{\vert t\vert}\right)\right]
\end{equation*}
near infinity, then $u$ must be constant.
Here $\rho_{\bm}(x)$ denotes the Riemannian distance from the boundary;
\item let $u:M\times (-\infty,0]\to \mathbb{R}$ be an ancient solution to the heat equation with Dirichlet boundary condition,
and $u_{\nu}\geq 0$ and $\partial_t u\leq 0$ over $\bm \times (-\infty,0]$.
If
\begin{equation*}
u(x,t)=o\left(\rho_{\bm}(x)+\sqrt{\vert t\vert}\right)
\end{equation*}
near infinity, then $u$ is constant. 
\end{enumerate}
\end{cor}

\begin{rem}
Due to the splitting theorem by Kasue \cite[Theorem C]{K2},
under the same setting as in Corollaries \ref{cor:Liouville1} and \ref{cor:Liouville2},
if $M$ is non-compact,
then it is isometric to $[0,\infty)\times \bm$.
\end{rem}

\section{Preliminaries}
Here,
let $(M,g)$ be an $n$-dimensional, complete Riemannian manifold with boundary.

\subsection{Laplacian comparison on manifolds with boundary}
We first recall a Laplacian comparison theorem for the distance function from the boundary.
The distance function from the boundary $\rho_{\bm}:M\to \mathbb{R}$ is defined as 
\begin{equation*}
\rho_{\bm}:=d(\cdot,\bm),
\end{equation*}
which is smooth outside of the cut locus for the boundary $\cut \bm$ (see e.g., \cite[Section 3]{S} for the basics of the cut locus for the boundary).

For $K,\Lambda \in \mathbb{R}$,
we denote by $s_{K,\Lambda}(t)$ the unique solution to the Jacobi equation $\varphi''(t)+K\varphi(t)=0$ with initial conditions $\varphi(0)=1$ and $\varphi'(0)=-\Lambda$.
We have the following Laplacian comparison theorem (see \cite[Corollary 2.44]{K1}):
\begin{thm}[\cite{K1}]\label{thm:comparison}
For $K,\Lambda \in \mathbb{R}$,
we assume $\ric_{M} \geq (n-1)K$ and $H_{\partial M}\geq (n-1)\Lambda$.
Then we have
\begin{equation*}
\Delta \rho_{\partial M}\leq (n-1)\frac{s'_{K,\Lambda}(\rho_{\partial M})}{s_{K,\Lambda}(\rho_{\partial M})}
\end{equation*}
outside of $\cut \bm$.
\end{thm}

\begin{rem}
For $K\geq 0$,
we see $s_{-K,-\sqrt{K}}(t)=e^{\sqrt{K}t}$;
in particular,
\begin{equation*}
\frac{s'_{-K,-\sqrt{K}}(t)}{s_{-K,-\sqrt{K}}(t)}=\sqrt{K}.
\end{equation*}
On the other hand,
we possess $s_{-K,0}(t)=\cosh\sqrt{K}t$;
in particular,
\begin{equation*}
\frac{s'_{-K,0}(t)}{s_{-K,0}(t)}=\sqrt{K}\frac{\sinh \sqrt{K}t}{\cosh \sqrt{K}t}\leq \sqrt{K}.
\end{equation*}
\end{rem}

\subsection{Reilly type formula}

We next recall the following useful formula,
which is used in the standard proof of the so-called Reilly formula (see e.g., \cite{R}, \cite[Chapter 8]{Li}):
\begin{prop}[\cite{R}]\label{prop:reilly}
For all $\varphi \in C^{\infty}(M)$ we have
\begin{align*}
\left(\Vert \nabla \varphi \Vert^{2}\right)_{\nu}=2\,\varphi_{\nu} \left[\Delta \varphi-\Delta_{\bm} (\varphi|_{\bm})-\varphi_{\nu} H  \right]&+2\,g_{\bm}\left(\nabla_{\bm}(\varphi|_{\partial M}),\nabla_{\bm}\varphi_{\nu} \right)\\
&-2\,\second \left(\nabla_{\bm}(\varphi|_{\partial M}),\nabla_{\bm}(\varphi|_{\partial M})  \right), \notag
\end{align*}
where $H$ and $\second$ are the mean curvature and second fundamental form,
respectively.
\end{prop}

\subsection{Harmonic functions}
We recall the following estimate for a harmonic function
\begin{equation*}
\Delta u=0,
\end{equation*}
which is a consequence of Bochner formula and refined Kato inequality (see e.g., \cite[(3.9)]{SY}):
\begin{prop}\label{prop:useful}
For $K\geq 0$,
we assume $\ric_M \geq -(n-1)K$.
Let $u:M\to (0,\infty)$ denote a positive harmonic function.
We set
\begin{equation}\label{eq:test}
\phi:=\frac{\Vert \nabla u\Vert}{u}.
\end{equation}
Then we have
\begin{equation*}
\Delta \phi\geq -(n-1)K \phi-\frac{2(n-2)}{n-1}\frac{g(\nabla \phi,\nabla u)}{u}+\frac{1}{n-1}\phi^3.
\end{equation*}
\end{prop}

\subsection{Heat equation}

We collect some properties of a positive solution to the heat equation
\begin{equation*}
\partial_t u=\Delta u
\end{equation*}
We possess the following (see e.g., \cite{SZ}, \cite[Lemma 4.1]{KS1}):
\begin{lem}\label{lem:easy}
We set
\begin{equation}\label{eq:log heat}
f:=\log u.
\end{equation}
Then we have
\begin{equation*}\label{eq:easy1}
(\Delta -\partial_{t}) f=-\Vert \nabla f \Vert^2.
\end{equation*}
\end{lem}

We also have the following (see e.g., \cite{SZ}, \cite[Lemma 4.2]{KS1}):
\begin{lem}\label{lem:simple}
Let $f$ be defined as $(\ref{eq:log heat})$.
We assume $\sup f<1$,
and set
\begin{equation}\label{eq:gradient log heat}
w:=\frac{\Vert \nabla f \Vert^2}{(1-f)^2}.
\end{equation}
Then we have
\begin{equation*}
\left(\Delta-\partial_{t}\right)w-\frac{2f g(\nabla w,\nabla f)}{1-f}\geq 2(1-f)w^2+\frac{2\ric(\nabla f,\nabla f)}{(1-f)^2}.
\end{equation*}
\end{lem}

\section{Proof of Theorem \ref{thm:main1}}

In this section,
we give a proof of Theorem \ref{thm:main1}.

\begin{proof}[Proof of Theorem \ref{thm:main1}]
Let $(M,g)$ and $u$ be as in Theorem \ref{thm:main1}.
We define a function $F:B_{R}(\bm)\to \mathbb{R}$ by
\begin{equation*}
F:=(R^2-\rho^2_{\bm})\phi,
\end{equation*}
where $\phi$ is defined as \eqref{eq:test}.
Assume that $F$ achieves its maximum at $\bar{x}\in B_{R}(\bm)$.

We first consider the case where $\bar{x}\in B_{R}(\bm) \setminus \partial M$.
Due to the well-known argument by Calabi \cite{C},
we may assume that $\bar{x}$ does not belong to $\cut \bm$ (see e.g., \cite[Theorem 1.1]{P} for the explicit construction of a barrier function).
At $\bar{x}$,
it holds that
\begin{equation*}
\nabla F=0,\quad \Delta F\leq 0,
\end{equation*}
and hence
\begin{equation*}
\frac{\nabla \rho^2_{\bm}}{R^2-\rho^2_{\bm}}=\frac{\nabla \phi}{\phi},\quad -\frac{\Delta \rho^2_{\bm}}{R^2-\rho^2_{\bm}}+\frac{\Delta \phi}{\phi}-\frac{2g(\nabla \rho^2_{\bm},\nabla \phi)}{(R^2-\rho^2_{\bm})\phi}\leq 0.
\end{equation*}
It follows that
\begin{equation}\label{eq:max}
\frac{\Delta \phi}{\phi}-\frac{\Delta \rho^2_{\bm}}{R^2-\rho^2_{\bm}}-\frac{2 \Vert \nabla \rho^2_{\bm}\Vert^2}{(R^2-\rho^2_{\bm})^2}\leq 0.
\end{equation}
By virtue of Theorem \ref{thm:comparison},
\begin{equation}\label{eq:maxapply}
\Delta \rho^2_{\bm}=2+2\rho_{\bm} \Delta \rho_{\bm}\leq 2+2(n-1)\sqrt{K}\rho_{\bm}\leq C_{n,1}(1+\sqrt{K}\rho_{\bm}),
\end{equation}
where $C_{n,1}:=2(n-1)$.
Since $\Vert \nabla \rho^2_{\bm} \Vert^2=4\rho^2_{\bm}$,
Proposition \ref{prop:useful}, \eqref{eq:max} and \eqref{eq:maxapply} yield
\begin{align*}
0&\geq \frac{\Delta \phi}{\phi}-\frac{    C_{n,1}(1+\sqrt{K}\rho_{\bm})  }{R^2-\rho^2_{\bm}}-\frac{8\rho^2_{\bm}}{(R^2-\rho^2_{\bm})^2}\\
  &\geq -(n-1)K-\frac{2(n-2)}{n-1}\frac{g(\nabla \phi,\nabla u)}{\phi u}+\frac{1}{n-1}\phi^2-\frac{    C_{n,1}(1+\sqrt{K}\rho_{\bm})  }{R^2-\rho^2_{\bm}}-\frac{8\rho^2_{\bm}}{(R^2-\rho^2_{\bm})^2}.
\end{align*}
Using
\begin{equation*}
\frac{g(\nabla \phi,\nabla u)}{\phi u}=\frac{2\rho_{\bm} g(\nabla \rho_{\bm},\nabla u ) }{(R^2-\rho^2_{\bm})u}\leq \frac{2\rho_{\bm}}{R^2-\rho^2_{\bm}}\phi,
\end{equation*}
we obtain
\begin{align*}
0&\geq \frac{F^2}{n-1}-\frac{4(n-2)\rho_{\bm} F}{n-1}-C_{n,1}(1+\sqrt{K}\rho_{\bm})(R^2-\rho^2_{\bm})-8\rho^2_{\bm}-(n-1)K(R^2-\rho^2_{\bm})^2\\
 &\geq \frac{F^2}{n-1}-2C_{n,2} R F-C_{n,1}(1+\sqrt{K}R)R^2-8R^2-(n-1)KR^4\\
 &\geq \frac{F^2}{n-1}-2C_{n,2} R F-C_{n,3}(1+\sqrt{K}R)^2R^2,
\end{align*}
where
\begin{equation*}
C_{n,2}:=2(n-2)(n-1)^{-1},\quad C_{n,3}:=\max \left\{C_{n,1}+8,\frac{C_{n,1}}{2},n-1\right\}.
\end{equation*}
Therefore,
\begin{equation*}
F(\bar{x})\leq (n-1)\left[ C_{n,2} R+\sqrt{C^2_{n,2}R^2+\frac{C_{n,3}}{n-1}R^2(1+\sqrt{K}R)^2   }    \right]\leq C_{n,4} R(1+\sqrt{K}R),
\end{equation*}
where
\begin{equation*}
C_{n,4}:=(n-1)\left(2C_{n,2}+\sqrt{\frac{C_{n,3}}{n-1}}\right).
\end{equation*}
On the ball we have
\begin{equation*}
\frac{3R^2}{4}\sup_{B_{R/2}(\bm)} \frac{\Vert \nabla u \Vert}{u}\leq C_{n,4} R(1+\sqrt{K}R).
\end{equation*}
We arrive at the desired inequality.

We next observe the case where $\bar{x}\in \partial M$.
In this case,
we refer to the argument of the proof of \cite[Theorem 2]{LY}.
At $\bar{x}$,
it holds that
\begin{equation*}
F_{\nu}\geq 0,
\end{equation*}
and hence
\begin{equation*}
\phi_{\nu}=\frac{F_{\nu}}{R^2}\geq 0;
\end{equation*}
in particular, $(\phi^2)_{\nu}\geq 0$.
Here we notice that
the Dirichlet boundary condition for $u$ and the assumption $u_{\nu}\geq 0$ yield $\Vert \nabla u\Vert=u_{\nu}$.
Since $\phi=\Vert \nabla \log u\Vert$, and $\log u$ also satisfies Dirichlet boundary condition,
Proposition \ref{prop:reilly} implies
\begin{align*}
0&\leq (\phi^2)_{\nu}=\left(\Vert \nabla \log u \Vert^{2}\right)_{\nu}=2(\log u)_{\nu}\left(\Delta \log u-(\log u)_{\nu}H \right)\\
  &=2\frac{u_{\nu}}{u}\left( -\frac{\Vert \nabla u \Vert^2}{u^2} -\frac{u_{\nu}}{u}H   \right)=2\frac{u^2_{\nu}}{u^2}\left( -\frac{u_{\nu}}{u} -H   \right).
\end{align*}
Therefore,
\begin{equation*}
\phi(\bar{x})=\frac{u_{\nu}}{u}\leq -H\leq (n-1)\sqrt{K}.
\end{equation*}
It follows that
\begin{equation*}
F(\bar{x})=R^2\phi(\bar{x})\leq (n-1)\sqrt{K}R^2\leq C_{n,5}R(1+\sqrt{K}R),
\end{equation*}
where $C_{n,5}:=n-1$.
Thus,
\begin{equation*}
\frac{3R^2}{4}\sup_{B_{R/2}(\bm)} \frac{\Vert \nabla u \Vert}{u}\leq C_{n,5} R(1+\sqrt{K}R).
\end{equation*}
We complete the proof.
\end{proof}

\section{Proof of Theorem \ref{thm:main2}}

In the present section,
we prove Theorem \ref{thm:main2}.
Let $(M,g)$ denote an $n$-dimensional, complete Riemannian manifold with compact boundary.

\subsection{Cutoff argument}
In this subsection,
we make preparations for the proof of Theorem \ref{thm:main2}.
To do so,
we first recall the following elementary fact:
\begin{lem}\label{lem:cutoff}
Let $\alpha \in (0,1)$.
Then there is a smooth function $\psi:[0,\infty)\times (-\infty,0]\to [0,1]$,
and a constant $C_\alpha>0$ depending only on $\alpha$ such that the following hold:
\begin{enumerate}\setlength{\itemsep}{3pt}
\item $\psi\equiv 1$ on $[0,R/2]\times [-T/4,0]$, and $\psi\equiv 0$ on $[R,\infty)\times (-\infty,-T/2]$;
\item $\partial_r \psi \leq 0$ on $[0,\infty)\times (-\infty,0]$, and $\partial_r \psi \equiv 0$ on $[0,R/2]\times (-\infty,0]$;
\item we have
\begin{equation*}
         \frac{\vert \partial_r \psi \vert}{\psi^\alpha}\leq \frac{C_{\alpha}}{R},\quad \frac{\vert \partial^2_{r}\psi\vert}{\psi^\alpha}\leq \frac{C_{\alpha}}{R^2},\quad \frac{\vert \partial_{t} \psi\vert}{\psi^{1/2}}\leq \frac{C}{T},
\end{equation*}
where $C>0$ is a universal constant.
\end{enumerate}
\end{lem}

Using Lemma \ref{lem:cutoff},
we show the following:
\begin{lem}\label{lem:difficult}
For $K \geq 0$,
we assume $\ric_{M} \geq -(n-1)K$ and $H_{\partial M}\geq 0$.
Let $u$ denote a positive solution to the heat equation on $Q_{R,T}(\bm)$.
We assume $u< 1$.
We define $f$ and $w$ as $(\ref{eq:log heat})$ and $(\ref{eq:gradient log heat})$ on $Q_{R,T}(\bm)$,
respectively.
We also take a function $\psi:[0,\infty)\times (-\infty,0]\to [0,1]$ in Lemma \ref{lem:cutoff} with $\alpha=3/4$,
and define
\begin{equation}\label{eq:cutoff function}
\psi(x,t):=\psi(\rho_{\bm}(x),t).
\end{equation}
Then we have
\begin{equation}\label{eq:difficult}
(\psi w)^2\leq \frac{C_{n,1}}{R^4}+\frac{\widetilde{C}}{T^2}+C_{n,2} K^2+2\Phi
\end{equation}
at every $(x,t )\in Q_{R,T}(\bm)$ with $x\notin \bm \cup \cut \bm$,
where for the universal constants $C_{3/4},C>0$ given in Lemma \ref{lem:cutoff},
we put
\begin{align*}\label{eq:dimension epsilon}
C_{n,1}&:=20C^2_{3/4}\left(\frac{(n-1)^2}{8}+\frac{1}{4}+C^2_{3/4}+\frac{27 C^2_{3/4}}{160}\right),\quad \widetilde{C}:=5C^2,\\
C_{n,2}&:=20(n-1)^2\left(1+\frac{C^2_{3/4}}{8}\right), \\
\Phi&:=(\Delta-\partial_{t})(\psi w)-\frac{2g\left( \nabla\psi, \nabla(\psi w) \right)}{\psi}-\frac{2f g(\nabla(\psi w),\nabla f)}{1-f}.
\end{align*}
\end{lem}
\begin{proof}
Notice that
$u< 1$ implies $f< 0$,
and hence $w$ is well-defined.
By direct computations and Lemma \ref{lem:simple},
\begin{align*}
\Phi&=\psi \left(\Delta-\partial_{t} \right)w-\frac{2\psi f g(\nabla w,\nabla f)}{1-f}+w\left(\Delta-\partial_{t}\right)\psi-\frac{2w\Vert \nabla \psi \Vert^2}{\psi}-\frac{2w f g(\nabla \psi,\nabla f)}{1-f}\\
                  &\geq 2(1-f)\psi w^2+\frac{2\psi \ric(\nabla f,\nabla f)}{(1-f)^2}+w\left(\Delta-\partial_{t} \right)\psi-\frac{2w\Vert \nabla \psi \Vert^2}{\psi}-\frac{2wf g(\nabla \psi,\nabla f)}{1-f}.
\end{align*}
It follows that
\begin{equation}\label{eq:main difficult}
2(1-f)\psi w^2\leq \Psi_1+\Psi_2+\Psi_3+\Psi_4+\Phi
\end{equation}
for
\begin{align*}
\Psi_1&:=-\frac{2\psi \ric(\nabla f,\nabla f)}{(1-f)^2},\quad \Psi_2:=-w \left(\Delta-\partial_{t}\right)\psi,\\
\Psi_3&:=\frac{2w\Vert \nabla \psi \Vert^2}{\psi},\quad \Psi_4:=\frac{2w f g(\nabla \psi,\nabla f)}{1-f}.
\end{align*}

We estimate $\Psi_1, \Psi_2, \Psi_3,\Psi_4$ from above.
The following Young inequality plays an essential role in the estimates:
For all $p,q\in (1,\infty)$ with $p^{-1}+q^{-1}=1$, $a,b\geq 0$, and $\epsilon>0$,
\begin{equation}\label{eq:Young}
ab\leq \frac{\epsilon a^p}{p}+\frac{b^q}{\epsilon^{q/p} q}.
\end{equation} 
The inequality
\begin{equation}\label{eq:grad est cutoff}
\frac{\Vert \nabla \psi \Vert^2}{\psi^{3/2}}\leq \frac{C^2_{3/4}}{R^2}
\end{equation}
is also useful,
which follows from Lemma \ref{lem:cutoff}.
We first deduce an upper bound of $\Psi_1$.
By the assumption for the Ricci curvature,
the Young inequality (\ref{eq:Young}) with $p,q=2$,
and $\psi\leq 1$,
\begin{equation}\label{eq:estimate1}
\Psi_1=-\frac{2\psi \ric(\nabla f,\nabla f)}{(1-f)^2}\leq 2(n-1)K \psi w\leq \epsilon \psi^2 w^2+\frac{(n-1)^2K^2}{\epsilon}\leq \epsilon \psi w^2+\frac{(n-1)^2K^2}{\epsilon}.
\end{equation}
We next give an upper bound of $\Psi_2$.
Due to Theorem \ref{thm:comparison},
\begin{align*}
\Psi_2&=-w \left(\Delta-\partial_{t}\right)\psi=-w  \left(           \partial_r \psi \Delta \rho_{\bm}+\partial^2_r \psi -\partial_{t} \psi \right)\\ \notag
                  & \leq w \vert \partial_r \psi \vert \Delta \rho_{\bm}+w \vert \partial^2_r \psi\vert +w \,\vert \partial_{t} \psi \vert\\
                  & \leq (n-1)\sqrt{K}w \vert \partial_r \psi \vert+w \vert \partial^2_r \psi\vert +w \,\vert \partial_{t} \psi \vert.
\end{align*}
From the Young inequality (\ref{eq:Young}) with $p,q=2$,
Lemma \ref{lem:cutoff},
and $\psi\leq 1$,
we derive
\begin{align}\label{eq:estimate2}
\Psi_2&\leq \left( \epsilon \psi w^2+\frac{(n-1)^2K}{4}\frac{ \vert \partial_r \psi \vert^2}{\epsilon \psi}   \right)+\left( \epsilon \psi w^2+\frac{1}{4}\frac{ \vert \partial^2_r \psi \vert^2}{\epsilon \psi}   \right)
+\left( \epsilon \psi w^2+\frac{1}{4}\frac{ \vert \partial_{t}  \psi\vert^2}{\epsilon \psi}   \right) \\ \notag
&\leq 3\epsilon  \psi w^2+\frac{C^2_{3/4}}{4\epsilon}\frac{ \psi^{1/2}}{R^4}+\frac{C^2}{4\epsilon}\frac{1}{T^2}+\frac{(n-1)^2 C^2_{3/4}}{4\epsilon} \frac{K}{R^2}\psi^{1/2}\\ \notag
&\leq 3\epsilon  \psi w^2+\frac{C^2_{3/4}}{4\epsilon}\frac{1}{R^4}+\frac{C^2}{4\epsilon}\frac{1}{T^2}+\frac{(n-1)^2C^2_{3/4}}{4\epsilon}\frac{K}{R^2}\\ \notag
&\leq 3\epsilon  \psi w^2+\frac{C^2_{3/4}}{4\epsilon}\left( 1+\frac{(n-1)^2}{2}   \right)\frac{1}{R^4}+\frac{C^2}{4\epsilon}\frac{1}{T^2}+\frac{(n-1)^2C^2_{3/4}}{8\epsilon}K^2,
\end{align}
where we used the inequality of arithmetic-geometric means in the last inequality for the $(K/R^2)$-term.
We next study an upper bound of $\Psi_3$.
By the Young inequality (\ref{eq:Young}) with $p,q=2$,
and (\ref{eq:grad est cutoff}),
\begin{equation}\label{eq:estimate3}
\Psi_3= \frac{2w\Vert \nabla \psi \Vert^2}{\psi} \leq \epsilon \psi w^2+\frac{\Vert \nabla \psi \Vert^4}{\epsilon \psi^3}\leq \epsilon \psi w^2+\frac{C^4_{3/4}}{\epsilon}\frac{1}{R^4}.
\end{equation}
We finally observe $\Psi_4$.
The Cauchy-Schwarz inequality,
the Young inequality (\ref{eq:Young}) with $p=4/3,q=4,\epsilon=4/3$,
and (\ref{eq:grad est cutoff}) lead us that
\begin{align}\label{eq:estimate4}
\Psi_4&=\frac{2w f g(\nabla \psi,\nabla f)}{1-f}\leq \frac{2w\vert f\vert \Vert \nabla \psi \Vert \Vert \nabla f \Vert}{1-f}=2 w^{3/2} \vert f\vert  \Vert \nabla \psi \Vert\\ \notag
                  &\leq (1-f)\psi w^2+\frac{27}{16} \frac{f^4}{(1-f)^3}  \frac{\Vert \nabla \psi \Vert^4}{\psi^3}\leq (1-f)\psi w^2+\frac{27C^4_{3/4}}{16} \,\frac{f^4}{(1-f)^3}\frac{1}{R^4}.
\end{align}

Combining (\ref{eq:estimate1}), (\ref{eq:estimate2}), (\ref{eq:estimate3}), (\ref{eq:estimate4}) with (\ref{eq:main difficult}),
we conclude
\begin{align*}
(1-f)\psi w^2 \leq 5\epsilon  \psi w^2&+\frac{C^2_{3/4}}{\epsilon}\left(\frac{(n-1)^2}{8}+\frac{1}{4}+C^2_{3/4}+  \frac{27\epsilon C^2_{3/4}}{16} \,\frac{f^4}{(1-f)^3}   \right)\frac{1}{R^4}+\frac{C^2}{4\epsilon}\frac{1}{T^2}\\
                                                        &+\frac{(n-1)^2}{\epsilon}\left(1+\frac{C^2_{3/4}}{8}\right)K^2+\Phi.
\end{align*}
We divide the both sides by $1-f$.
Since $1/(1-f)\leq 1$ and $f/(1-f)\leq 1$,
we see
\begin{align*}
(1-5\epsilon)\psi w^2 &\leq \frac{C^2_{3/4}}{\epsilon}\left(\frac{(n-1)^2}{8}+\frac{1}{4}+C^2_{3/4}+\frac{27\epsilon C^2_{3/4}}{16}\right)\frac{1}{R^4}+\frac{C^2}{4\epsilon}\frac{1}{T^2}\\
&\qquad \qquad \qquad \qquad \,+\frac{(n-1)^2}{\epsilon}\left(1+\frac{C^2_{3/4}}{8}\right)K^2+\Phi.
\end{align*}
By letting $\epsilon \to 1/10$,
we obtain
\begin{equation*}
\psi w^2 \leq \frac{C_{n,1}}{R^4}+\frac{\widetilde{C}}{T^2}+C_{n,2} K^2+2\Phi.
\end{equation*}
Since $(\psi w)^2\leq \psi w^2$,
we arrive at the desired one (\ref{eq:difficult}).
\end{proof}

\subsection{Proof}

We are now in a position to prove Theorem \ref{thm:main2}.
\begin{proof}[Proof of Theorem \ref{thm:main2}]
Let $(M,g)$ and $u$ be as in Theorem \ref{thm:main2}.
We may assume that $A=1$;
in particular, $u<1$ on $Q_{R,T}(\bm)$.
We define functions $f,w$ and $\psi$ as $(\ref{eq:log heat}), (\ref{eq:gradient log heat})$ and (\ref{eq:cutoff function}),
respectively.
We see $f<0$ on $Q_{R,T}(\bm)$.
In the case where $u$ is constant on $Q_{R/2,T/4}(\bm)$,
we have nothing to do,
and we may assume that $u$ is non-constant on $Q_{R/2,T/4}(\bm)$;
in particular,
$\psi w$ is positive at a point in $Q_{R/2,T/4}(\bm)$.
We take a maximum point $(\bar{x},\bar{t})$ of $\psi w$ in $Q_{R,T}(\bm)$.

We first show $\bar{x} \notin \bm$ by contradiction.
We suppose $\bar{x}\in \partial M$.
Then at $(\bar{x},\bar{t})$,
\begin{equation*}
(\psi w)_{\nu}\geq 0,
\end{equation*}
and hence
\begin{equation*}
\psi_{\nu}w+\psi w_{\nu}=\psi w_\nu \geq 0;
\end{equation*}
in particular, $w_{\nu}\geq 0$.
Since $w=\Vert \nabla \log (1-f)\Vert^2$, and $\log (1-f)$ also satisfies the Dirichlet boundary condition,
Proposition \ref{prop:reilly} implies
\begin{equation*}
0\leq w_{\nu}=\left(\Vert \nabla \log (1-f) \Vert^{2}\right)_{\nu}=2(\log (1-f))_{\nu}\left(\Delta \log (1-f)-(\log (1-f))_{\nu}H \right).
\end{equation*}
We now possess $\Vert \nabla f \Vert=f_{\nu}$ since $f$ satisfies the Dirichlet boundary condition,
and $f_{\nu}=u_{\nu}/u\geq 0$.
Therefore,
Lemma \ref{lem:easy} yields
\begin{align*}
\Delta \log (1-f)&=-\frac{\Vert \nabla f\Vert^2}{(1-f)^2}-\frac{\Delta f}{1-f}=-\frac{\Vert \nabla f\Vert^2}{(1-f)^2}+\frac{\Vert \nabla f\Vert^2}{1-f}-\frac{\partial_t f}{1-f}\\
&=-f\frac{\Vert \nabla f\Vert^2}{(1-f)^2}-\frac{\partial_t f}{1-f}=-f\frac{f^2_{\nu}}{(1-f)^2}-\frac{\partial_t f}{1-f}=-f (\log (1-f))^2_{\nu}-\frac{\partial_t f}{1-f}.
\end{align*}
Moreover,
\begin{equation*}
(\log (1-f))_{\nu}=-\frac{f_{\nu}}{1-f}=-\frac{\Vert \nabla f\Vert}{1-f}=-w^{1/2}.
\end{equation*}
From $\partial_t f=\partial_t u/u\leq 0$, it follows that
\begin{equation*}
0\leq -2w^{1/2}\left(-fw-\frac{\partial_t f}{1-f}+w^{1/2}H \right)\leq -2w^{1/2}\left(-fw+w^{1/2}H \right).
\end{equation*}
We conclude
\begin{equation*}
-fw^{1/2}\leq  -H\leq 0
\end{equation*}
By $-f>0$,
we see $w=0$ at $(\bar{x},\bar{t})$.
This means
$\psi w\equiv 0$ on $Q_{R,T}(\bm)$.
This is a contradiction.

In view of the Calabi argument,
we may assume that $\bar{x} \notin \bm \cup \cut \bm$.
By Lemma \ref{lem:difficult},
\begin{equation}\label{eq:using difficult}
(\psi w)^2\leq c_n \left(\frac{1}{R^4}+\frac{1}{T^2}+K^2 \right)+2\Phi
\end{equation}
at $(\bar{x},\bar{t})$ for
\begin{equation*}
c_{n}:=\max \left\{C_{n,1},\widetilde{C},C_{n,2} \right\}.
\end{equation*}
On the other hand,
since $(\bar{x},\bar{t})$ is a maximum point,
it holds that
\begin{equation*}
\Delta(\psi w)\leq 0,\quad \partial_{t} (\psi w)\geq 0,\quad \nabla (\psi w)=0
\end{equation*}
at $(\bar{x},\bar{t})$;
in particular, $\Phi(\bar{x},\bar{t})\leq 0$.
Therefore,
(\ref{eq:using difficult}) implies
\begin{equation*}\label{eq:smooth conclusion}
(\psi w)(x,t) \leq (\psi w)(\bar{x},\bar{t})\leq c^{1/2}_{n}\left(\frac{1}{R^4}+\frac{1}{T^2}+K^2     \right)^{1/2}\leq c^{1/2}_{n} \left(\frac{1}{R^2}+\frac{1}{T}+K\right)
\end{equation*}
for all $(x,t)\in Q_{R,T}(\bm)$.
By $\psi \equiv 1$ on $Q_{R/2,T/4}(\bm)$,
and by the definition of $w$ and $f$,
\begin{equation*}
\frac{\Vert \nabla u \Vert}{u} \leq c^{1/4}_{n} \left(\frac{1}{R}+\frac{1}{\sqrt{T}}+\sqrt{K}\right)\left( 1+\log \frac{1}{u} \right)
\end{equation*}
on $Q_{R/2,T/4}(\bm)$.
Thus,
we complete the proof of Theorem \ref{thm:main2}.
\end{proof}

\subsection*{{\rm Acknowledgements}}
The second named author is grateful to Professor Kazuhiro Kuwae for useful discussion.
The first named author was supported by JSPS KAKENHI (JP19K14521).
The second named author was supported by JSPS Grant-in-Aid for Scientific Research on Innovative Areas ``Discrete Geometric Analysis for Materials Design" (17H06460).


\end{document}